\documentclass[a4paper,10pt]{article}
\usepackage{anysize}
\usepackage{amsfonts,amsmath,amssymb,amsthm}
\usepackage{mathrsfs}
\usepackage{url,hyperref,verbatim,nicefrac}
\usepackage{color,graphicx,graphics}
\usepackage{latexsym, paralist}

\newcommand{\defeq}{\mathrel{\overset{\makebox[0pt]{\mbox{\normalfont\tiny\sffamily def}}}{=}}}

\newtheorem{thm}{Theorem}[section]
\newtheorem{thmnonumber*}{Theorem}

\newtheorem{theorem}[thm]{Theorem}

\newtheorem{lemma}[thm]{Lemma}
\newtheorem{proposition}[thm]{Proposition}
\newtheorem{corollary}[thm]{\bf Corollary}

\theoremstyle{definition}
\newtheorem{definition}[thm]{Definition}

\newtheorem{example}[thm]{Example}

\newtheorem{remark}[thm]{Remark}

\newcommand{\K}{\mathbb{K}}

\newcommand{\PP}{\mathbb{P}}

\newcommand{\mm}{\mathfrak{m}}
\newcommand{\nn}{\mathfrak{n}}
\newcommand{\pp}{\mathfrak{p}}

\newcommand{\reg}{\operatorname{reg}}

\newcommand{\Proj}{\operatorname{Proj}}

\newcommand{\OOO}{\mathscr{O}}
\newcommand{\III}{\mathscr{I}}

\date{February 9, 2017}

\begin{document}

\title{Regulating Hartshorne's connectedness theorem \\}
\vskip7mm
\author{Bruno Benedetti \thanks{Supported by NSF Grant  1600741, ``Geometric Combinatorics and Discrete Morse Theory'', and by the DFG Collaborative Research Center TRR 109, ``Discretization in Geometry and Dynamics''.}\\
\small Univ.~of Miami, USA\\
\small bruno@math.miami.edu\\
\and Barbara Bolognese \\
\small Univ.~of Sheffield, United Kingdom\\
\small b.bolognese@sheffield.ac.uk\\ \phantom{A}
\and Matteo Varbaro \thanks{Supported by PRIN  2010S47ARA\_003 ``Geometria delle Variet\`a Algebriche".}\\
\small Univ.~degli Studi di Genova, Italy\\
\small varbaro@dima.unige.it\\}
\maketitle
\vskip9mm

\begin{abstract}
A classical theorem by Hartshorne states that the dual graph of any arithmetically Cohen--Macaulay projective scheme is connected.
We give a quantitative version of Hartshorne's result, in terms of Castelnuovo--Mumford regularity. If $X \subset \mathbb{P}^n$ is an arithmetically Gorenstein projective scheme of regularity $r+1$, and if every irreducible component of $X$ has regularity $\le r'$, we show that the dual graph of $X$ is  
$\lfloor{\frac{r+r'-1}{r'}}\rfloor$-connected. The bound is sharp.

We also provide a strong converse to Hartshorne's result: Every connected graph is the dual graph of a suitable arithmetically Cohen--Macaulay projective curve of regularity $\le 3$, whose components are all rational normal curves. The regularity bound is smallest possible in general.

Further consequences of our work are:
\begin{compactenum}[(1)]
\item Any graph is the Hochster--Huneke graph of a complete equidimensional local ring. (This answers a question by Sather--Wagstaff and Spiroff.)
\item The regularity of a curve is not larger than the sum of the regularities of its primary components.\\
\end{compactenum}
\end{abstract}

\vskip9mm

\section{Introduction}

Intersection patterns of projective curves are a classical topic in algebraic geometry. 
The \emph{dual graph} of a scheme is obtained by taking as many vertices as the irreducible components, and by connecting two distinct vertices with a single edge whenever the two corresponding components intersect in a subscheme of dimension one less than the scheme.

In 1962, Hartshorne showed that arithmetically Cohen--Macaulay projective schemes have connected dual graphs~\cite[Theorem 18.12]{Ei}. (The original statement by Hartshorne is slightly more general, cf.~Remark \ref{rem:Ha} below). This result, henceforth called \textsc{Hartshorne's Connectedness Theorem}, triggers two nontrivial questions:

\vskip2mm
\begin{compactenum}
\item \emph{(Inverse problem)} Do all connected graphs arise this way? If so, how to reconstruct a (nice) projective scheme from its intersection pattern?
\item \emph{(Quantitative problem)} Under extra algebraic parameters, can we strengthen the combinatorial conclusion quantitatively? (e.g.\ can we bound the graph diameter, connectivity, or expansion?)
\end{compactenum}

\vskip2mm
\enlargethispage{-3mm}

It is not difficult to show that any connected graph is indeed dual to some algebraic curve, as elegantly explained in Koll\`ar \cite{Kollar}. Here we prove the following:

\vspace{8mm}

\noindent{\bf Theorem} (Theorem \ref{thm:CMfromGraph})
Any connected graph $G$ is dual to an arithmetically Cohen--Macaulay reduced curve $C_G \subset \mathbb{P}^n$, of regularity $\le 3$, with the additional property that all irreducible components of $C_G$ are rational normal curves (in their span). In addition, all of the singular points of $C_G$ have multiplicity 2 and $C_G$ is locally a complete intersection. 
\vspace{3mm}

The embedding $C_G\subset \PP^{n}$ we provide is explicit, and optimal in two aspects. First of all, if $C\subset \PP^n$ is a projective curve of regularity 2 whose points have multiplicity $\leq 2$, then the dual graph of $C$ must be a tree, cf.~Remark \ref{small schemes}. So regularity 3 is the smallest possible in general. Secondly, realizing every graph with irreducible components of regularity $1$ is not possible, since some graphs are not dual to any line arrangement, cf.~Proposition~\ref{prp:hierarchy}. Here we realize the components with rational normal curves, which have regularity $\le 2$.

As for the quantitative problem, some progress was obtained in 2014 by the first and third author, who proved that if $X \subset \mathbb{P}^n$ is an arithmetically Gorenstein (reduced) subspace arrangement  of regularity $r+1$, then the dual graph of $X$ is  $r$-connected \cite[Theorem\ 3.8]{BV}.  

Can one extend this result from subspace arrangements to arbitrary projective schemes? At first, the answer seems negative: For example, there are arithmetically Gorenstein curves of high regularity whose dual graph is a path, so not even $2$-connected, cf.~\cite[Examples 3.4 and 5.10]{BV}. 

In the present paper we bypass these difficulties and show that the result does extend; the conclusion ``$r$-connected'' should be replaced with ``$\lfloor{\frac{r+r'-1}{r'}}\rfloor$-connected'', where $r'$ is the maximum regularity of a primary component of $X$. Under the additional assumption that $X$ is reduced, the theorem works also if $r'$ is the maximum degree of an irreducible component of $X$:

\vspace{4mm}

\noindent{\bf Main Theorem} (Theorem \ref{thm:1} \& Corollary \ref{cor:1degreeversion})
Let $X$ be an arithmetically Gorenstein projective scheme of regularity $r+1$. 
\begin{compactenum}[\rm (A)]
\item If every primary component of $X$ has regularity $\le r'$, the dual graph of $X$ is  $\lfloor{\frac{r+r'-1}{r'}}\rfloor$-connected.
\item If every irreducible component of $X$ has degree $\le D$, and $X$ is reduced, the dual graph of $X$ is  $\lfloor{\frac{r+D-1}{D}}\rfloor$-connected.
\end{compactenum}

\vspace{3mm}
One key to this result is a subadditivity lemma: We show that the regularity of a curve is not larger than the sum of the regularities of its primary components (Lemma \ref{lem:EGrevisited}). For line arrangements this follows by the work of Derksen and Sidman \cite{DS}. Unlike Derksen--Sidman's bound, though, our bound does not extend to higher dimensions: Compare Example \ref{ex:aldo}. 

Since every complete intersection $X\subset \PP^n$, defined by equations $f_1,\ldots ,f_c$, is arithmetically Gorenstein of regularity $\deg f_1+\ldots +\deg f_c-c+1$, our Main Theorem yields a rigidity condition for the possible configurations of the irreducible components of a complete intersection (cf.~e.g.~Corollary \ref{cor:ci}).

For subspace arrangements, that is when $r'=1$, the bound of the Main Theorem is sharp by~\cite[Example 3.13]{BV}. With some computational effort, we are able to provide examples of non-linear arrangements (i.e. $r'>1$) where the bound is still sharp, cf.\ Section \ref{sec:sharpness}.

\begin{remark} \label{rem:Ha}
In Hartshorne's connectedness theorem, the \emph{assumption} ``$X\subseteq \PP^n$ is arithmetically Cohen-Macaulay" depends on the embedding, while the \emph{conclusion} on the connectivity of the dual graph of $X$ does not. 
It is worth mentioning that Hartshorne's original result \cite[Theorem 2.2]{Ha} assumes a more general condition on $X$ that \emph{is} intrinsic, namely, ``$X$ is a connected projective scheme such that $\mathcal{O}_{X,x}$ satisfies Serre's condition $S_2$ for all $x\in X$''. Any arithmetically Cohen-Macaulay projective scheme of positive dimension is connected (and satisfies $S_2$ locally). 
For simplicity, we preferred to state Hartshorne's result in Eisenbud's version~\cite[Theorem 18.12]{Ei}. Moreover, both in \cite[Theorem 2.2]{Ha} and in \cite[Theorem 18.12]{Ei}, the conclusion is that removing a subset of codimension $\geq 2$ will not disconnect $X$. This is equivalent to say that the dual graph of $X$ is connected by \cite[Proposition 1.1]{Ha}.
\end{remark}

\section{Glossary} In the present paper, all fields are assumed to be infinite. 

The {\it Castelnuovo-Mumford regularity} of a projective scheme $X\subset \PP^n$ over a field $\K$, denoted by $\reg X$, is the least integer $k$ such that $H^i(X,\III_X(k-i))=0$ for all  $i\geq 1$, where $\III_X\subset \OOO_{\PP^n}$ is the sheaf of ideals associated to the embedding $X\subset \PP^n$. If $S \defeq \K[x_0,\ldots ,x_n]$ and $\mm$ is the irrelevant ideal of $S$, we denote by $I_X$ the unique saturated ideal of $S$ such that $X=\Proj(S/I_X)$; in other words, 
\[I_X=\oplus_{z\in \mathbb{Z} } \, \Gamma(X,\III_X(z)).\] 
The Castelnuovo--Mumford regularity of a finitely generated $\mathbb{Z}$-graded $S$-module $M$, is defined by
\[\reg M \defeq \max\{i+j:H_{\mm}^i(M)_j\neq 0\}=\max \{j-i:\mathrm{Tor}_i^S(M,\K)_j\neq 0\}.\]
The two definitions are compatible: If $X\subset \PP^n$ is a projective scheme, then $I_X$ is an $S$-module and one has $\reg X=\reg I_X$. For further details, see e.g.\ \cite{Eisenbud}.

We say that $X \subset \mathbb{P}^n$ is {\it arithmetically Cohen-Macaulay} (resp. {\it arithmetically Gorenstein}) if $S/I_X$ is a Cohen-Macaulay (resp. Gorenstein) ring. We say that $X \subset \mathbb{P}^n$ is \emph{locally Cohen-Macaulay} (resp. \emph{locally Gorenstein}) if the stalk $\OOO_{X,x}$ is a Cohen-Macaulay (resp. Gorenstein) local ring for any $x\in X$. For both the Cohen--Macaulay and the Gorenstein property, ``arithmetically'' is much stronger than ``locally''.

We say that $X=X_1\cup X_2 \cup \ldots \cup X_s$ is a {\it primary decomposition} of $X$ if $I_X=I_{X_1}\cap I_{X_2}\cap \ldots \cap I_{X_s}$ is a primary decomposition of $I_X$. The $X_i$'s are called \emph{primary components}; in this paper, by \emph{irreducible components} we mean the reduced schemes associated to the primary components. Primary decompositions need not be unique. However, they are unique if $X$ has no embedded components; this is always the case if $X\subset \PP^n$ is arithmetically Cohen-Macaulay. 

If $X_1,\ldots ,X_s$ are the irreducible components of $X$, the {\it dual graph} $G(X)$ is the graph  whose vertices are $\{1,\ldots ,s\}$ and such that $\{i,j\}$ is an edge if and only if $X_i\cap X_j$ has dimension $\dim X-1$. All graphs considered in this paper are simple, i.e.\ without loops or multiple edges.  As the one-point graph is trivial to handle (it is dual to $\PP^1$, for example), we will only consider graphs with at least two vertices. The \emph{degree} $\deg v$ of a vertex $v$ of $G$  is the number of edges containing $v$.
A graph $G$ is called \emph{$k$-connected} (with $k$ a positive integer) if $G$ has at least $k+1$ vertices, and the deletion of less than $k$ vertices from $G$, however chosen, does not disconnect it.

\section{From graphs to curves}
For this section, $\K$ will be algebraically closed. Let $G$ be a graph on $s>1$ vertices, labeled by $1, \ldots, s$, and let $E(G)$ be the set of edges of $G$. We learned the following argument to produce a projective curve whose dual graph is $G$ from \cite{Kollar}. Pick $s$ distinct lines $L_1,\ldots ,L_s\subset \PP^2$ such that no three of them meet at a common point, and set $P_{ij} \defeq L_i\cap L_j$ for all $i\neq j$.

 Let $X$  be the blow-up  of $\PP^2$ along $\bigcup_{\{i,j\}\notin E(G)}P_{ij}$, and let $C_G$ be the strict transform of $\bigcup_{i=1}^sL_i$. By construction, $C_G$ is a projective curve whose dual graph is $G$. We also denote by $C_i$ the strict transform of $L_i$ for any $i=1,\ldots ,s$. By the blow-up closure lemma, $C_G$ is isomorphic to the blow-up of $\bigcup_{i=1}^sL_i$ along $\bigcup_{\{i,j\}\notin E(G)}P_{ij}$.

Since $\bigcup_{i=1}^sL_i$ is reduced and locally a complete intersection, and since we are blowing up only ordinary double points, the curve $C_G$ is also reduced and locally a complete intersection. 
The goal of this section is to describe an embedding of $C_G$ that is arithmetically Cohen--Macaulay if $G$ is connected. In fact, we will see that much more is true.

\bigskip
\textbf{The embedding.} Let us write each line $L_i\subset \PP^2$ as $\ell_i=0$, for a linear form $\ell_i \in S=\K[x,y,z]$. The condition  that no three of the $L_i$ meet at a common point, means that any three of the $\ell_i$ are linearly independent.
The defining ideal of $\bigcup_{\{i,j\}\notin E(G)}P_{ij}\subset \PP^2$ is:
\[I \defeq  \bigcap_{\{i,j\}\notin E(G)}(\ell_i,\ell_j)  \, \subset \, S.\]
For $d \in \mathbb{N}$, let $I_d$ be the $\K$-vector space of the degree-$d$ elements of $I$, $R[d]$ the $\K$-subalgebra of $S$ generated by $I_d$, and
\[A[d] \defeq \frac{R[d]}{(\ell_1\ell_2\cdots \ell_s)\cap R[d]}.\]
Finally, for any $\mathbb{Z}$-graded ring $T$ and any positive integer $e$, we denote by $T^{(e)}=\bigoplus_{k\in\mathbb{Z}}T_{ke}$ the $e$-th Veronese of $T$.

\begin{theorem} \label{thm:CMfromGraph}
Let $G$ be a connected graph on $s$ vertices. With the  notation above, assume that $d \ge \binom{s-1}{2}$. Then
\begin{compactenum}[\rm (i)]
\item $\Proj(A[d])=C_G$; 
moreover, $\reg A[d] \le |E(G)|-s+2$.
\item 
For each $i=1, \ldots, s$, the irreducible component $C_i$ of $\Proj (A[d])$ corresponding to the vertex $i$ of $G$ is a rational normal curve (in its span) of degree $\deg i+d-s+1$.
\item If in addition $e\ge \reg A[d]$,
then $A[d]^{(e)}$ is Cohen--Macaulay; moreover, the regularity of $\Proj(A[d]^{(e)})$ is equal to $2$ if $G$ is a tree, and to $3$ otherwise; and the irreducible components of $\Proj(A[d]^{(e)})$ are rational normal curves of degree $e(\deg i+d-s+1)$.
\end{compactenum}
\end{theorem}

\begin{proof}
\textsc{Part} (i): Notice that $\binom{s-1}{2} = \binom{s}{2}-s+1$ is greater than or equal to the degree of $\bigcup_{\{i,j\}\notin E(G)}P_{ij}$, which is $\binom{s}{2}-|E(G)|$. The latter is certainly bigger than the regularity of $I$, since $I$ is an ideal of points. In particular $d\ge \binom{s-1}{2}$ is bigger than the highest degree of a minimal generator of $I$. This implies that $R[d]$ is a coordinate ring of the blow-up $X$; for a proof of such implication see for example \cite[Lemma 1.1]{CH}. As a consequence, if $d\ge \binom{s-1}{2}$, then $A[d]$ is a coordinate ring of $C_G$. 
Moreover, the degree of the strict transform of $L_i$, with respect to the embedding given by $A[d]$, is $d-s+1+\deg i$. 
Therefore, 
\[\deg A[d] = \sum_{i=1}^s (d-s+1+\deg i) = s (d-s+1) + \sum_{i=1}^s \deg i =  sd-s^2+s+2|E(G)|.\] Since $\Proj(A[d])$ is a reduced connected projective curve, the Eisenbud-Goto conjecture holds \cite{Giaimo}, so 
\begin{equation}\label{eq1}
\reg(A[d])\le \deg A[d] -\dim \mathrm{Span} \Proj(A[d]) +1 = sd-s^2+2|E(G)|+s +2-\dim_\mathbb{K}A[d]_d.
\end{equation}
Yet $A[d]$ is isomorphic to the $\K$-subalgebra $B[d]$ of $S/(\ell_1\ell_2\cdots \ell_s)$ generated by the degree-$d$ part of the ideal $I/(\ell_1\ell_2\cdots \ell_s)$. So:
\begin{equation}\label{eq2}
\dim_\mathbb{K}A[d]_d=\dim_\mathbb{K}B[d]_d=\dim_\mathbb{K}\left (\frac{I}{(\ell_1\ell_2\cdots \ell_s)} \right)_d=\dim_\mathbb{K}\left(\frac{S}{(\ell_1\ell_2\cdots \ell_s)}\right)_d-\dim_\mathbb{K}\left(\frac{S}{I}\right)_d.
\end{equation}
\noindent The Hilbert function of complete intersections is well-known: Since $d\ge \binom{s-1}{2} \ge s-2 =\reg(\ell_1\ell_2\cdots \ell_s)-2$,  
\begin{equation}\label{eq3}
\dim_\mathbb{K}\left(\frac{S}{(\ell_1\ell_2\cdots \ell_s)}\right)_d=sd-\frac{s^2}{2}+\frac{3s}{2}.
\end{equation}
\noindent Also, since $d\ge \binom{s}{2}-s+1\geq \reg I-1$, the dimension of $(S/I)_d$ is the number of points defined by $I$, so
\begin{equation}\label{eq4}
\dim_\mathbb{K}\left(\frac{S}{I}\right)_d=\binom{s}{2}-|E(G)|.
\end{equation}
Putting together Equations \eqref{eq1}, \eqref{eq2}, \eqref{eq3} and \eqref{eq4}, we conclude. 

\medskip
\noindent \textsc{Part}  (ii):  For each $i$, notice that the coordinate ring of  the strict transform $C_i \subset \Proj(R[d])$ of the line $L_i$ is isomorphic to the $\K$-subalgebra of $S/(\ell_i)$ generated by the degree-$d$ part of the ideal $\frac{I+(\ell_i)}{(\ell_i)}$. Such an ideal is generated by a homogeneous polynomial $f$; explicitly, $f$ is the (image of) the product of the $\ell_j$ such that $\{i,j\}\notin E(G)$. Of course, $d \ge \binom{s-1}{2} \ge \deg(f)=s-1-\deg i$, so in this embedding $C_i \subset \Proj(A[d])$ is a rational normal curve of degree $d-s+\deg i+1$ in its span.

\medskip
\noindent \textsc{Part}  (iii): It is well known that $\Proj (A[d]^{(e)})= \Proj (A[d])$ for any positive integer $e$. We shall freely make use of the following graded iso\-morphism relating the local cohomology of $A[d]$ with that of its $e$-th Veronese, cf.~\cite[Theorem 3.1.1]{GW}: 
\[H_{\nn'}^i(A[d]^{(e)})\cong \bigoplus_{k\in\mathbb{Z}}H_{\nn}^i(A[d])_{ke} \ \ \ \forall \ i\in \mathbb{N},\]
where by $\nn$ and $\nn'$ we denote the irrelevant ideals of $A[d]$ and $A[d]^{(e)}$, respectively.

Since $A[d]$ is reduced, we have $H^0_{\nn}(A[d])=0$. Furthermore, for any positive integer $k$ we have $H^1_{\nn}(A[d])_{-k}=H^0(\Proj(A[d]),\mathcal{O}_{\Proj(A[d])}(-k))=0$, since negative twists of an ample line bundle over $C_G$ do not have global sections. By definition of regularity, $H^1_{\nn}(A[d])_k=0$ also for $k\geq \reg(A[d])$. 
Finally, since $A[d]$ is reduced and $\K$ is algebraically closed, the connectedness of $G$ implies  $H^1_{\nn}(A[d])_0=0$ .
Thus after applying an $e$-th Veronese with $e\geq \reg(A[d])$, all the undesired nonzero cohomologies disappear; the Cohen--Macaulay property of $A[d]^{(e)}$ follows.  
Since $\dim (A[d]) = 2$, the regularity of the $e$-th Veronese, with $e\ge \reg(A[d])-1$, is at most $2$; also, it is at least $1$, because $A[d]^{(e)}$ is not a polynomial ring (we are assuming $s>1$). So the regularity of the $e$-th Veronese is equal to $2$ if and only if $H^2_{\nn}(A[d]^{(e)})_0\neq 0$. Let us consider the arithmetic genus  $p_a(C_G) = \dim_{\K} H^2_{\nn}(A[d])_0$. This integer does not depend on the embedding of $C_G$. It is related to the arithmetic genus of $H \defeq \bigcup_{i=1}^sL_i$ via the following formula:
\[p_a(C_G)=p_a(H)-\sum_{\{i,j\}\notin E} \binom{\mu_{P_{ij}}(H)}{2},\]
where $\mu_{P_{ij}}(H)$ is the multiplicity of $P_{ij}$ on $H$. (See e.g.\ \cite[Theorem 5.9]
{Perrin} for a proof of the formula in the irreducible case; the same proof works also for reducible curves with same number of connected components; the two curves $C$ and $X$ of Perrin's notation are in our case the curves $H$ and $C_G$, respectively.) But all such multiplicities are 2, and $p_a(H)=\binom{s-1}{2}$, so the above equation can be rewritten as:
\[p_a(C_G)=\binom{s-1}{2}-\binom{s}{2}+|E(G)| = 1-s+|E(G)|.\]
Yet any connected graph has at least $s-1$ edges, with equality for trees. So $G$ is not a tree if and only if $p_a(C_G)=\dim_{\K}H^2_{\nn}(A[d]^{(e)})_0> 0$, if and only if $\reg (A[d]^{(e)}) = 2$. The last claim follows from the fact that the $e$-th Veronese sends rational normal curves to rational normal curves, while the degree gets multiplied by a factor $e$. 
\end{proof}

\begin{remark}\label{small schemes}
In the projective curve $C_G$, every point belongs to at most two irreducible components. If a connected projective curve $C\subset \PP^n$ has this property, and in addition it has Castelnuovo--Mumford regularity 2, then its dual graph must be a tree. (So if $G$ is not a tree, the ``regularity $\le$ 3'' result of Theorem \ref{thm:CMfromGraph}, (iii), is best possible.) To prove this, note that such a $C\subset \PP^n$ would be a small scheme in the sense of \cite{EGHP}. By \cite[Theorem 0.4]{EGHP}, there exists an ordering $C_1,\ldots ,C_s$ of the irreducible components of $C$ such that $(C_1\cup C_2 \ldots \cup C_i)\cap C_{i+1}$ 
is a single point for all $i=1,\ldots ,s-1$. But this is possible if and only if the dual graph is a tree. Note that if $G$ is a tree, the $e$ in the statement of Theorem \ref{thm:CMfromGraph} can be chosen to be 1.
\end{remark}

Sometimes it is possible to improve on the bounds for $d$ and $e$ given by  Theorem \ref{thm:CMfromGraph}. Here is one example where one can save $1$ both on $d$ and~$e$:

\begin{example} \label{ex:K4minusEdge}
Let $G_0$ be $K_4$ minus an edge, that is, the graph 
$G_0 = 13, \; 14, \; 23, \; 24, \; 34.$
In $\K[x,y,z]$ let us pick the four linear forms $\ell_1 = x$, $\; \ell_2= y$, $\; \ell_3 = z$, and $\ell_4= x + y + z$.
Since the missing edge in $G_0$ is $12$, we need to blow up $\mathbb{P}^2$ at the ideal $I=(\ell_1, \ell_2) = (x, y)$. We seek a $d$ such that $A[d]$ is a coordinate ring of $C_{G_0}$. Theorem \ref{thm:CMfromGraph} guarantees that any $d \ge 3$ would work. In fact, using directly \cite[Lemma 1.1]{CH}, we see that $d=2$ works already. We have
\[ A[2] =  \displaystyle \frac{\K[x^2,\, x y, \, x z, y^2, \,y z]}{
\left( \, x  y z (x + y + z) \right)}\cong \frac{\K[y_0,\ldots ,y_4]}{\left( \ y_0y_3-y_1^2, \ y_0y_4-y_1y_2, \ y_1y_4-y_2y_3, \ y_2(y_1+y_3+y_4) \ \right)}. \]
The Macaulay2 software \cite{macaulay2} allows to compute $\reg(A[2])=2$, so the proof of Theorem \ref{thm:CMfromGraph}, part (iii), guarantees that the second Veronese of $A[2]$ is Cohen--Macaulay. But in fact, $A[2]$ is already Cohen-Macaulay, so no Veronese is needed. 
En passant, note that 
 $A[2]^{(e)}$ is not Gorenstein for any $e$, because  the Veronese of any non-Gorenstein ring is not Gorenstein \cite[Theorem 3.2.1]{GW}.
\end{example}

\begin{remark}
The graph $G_0$ of Example \ref{ex:K4minusEdge} is also realizable as dual graph of an arrangement of projective lines. In fact, it is even a \emph{line (intersection) graph}, i.e.\ the dual graph of another graph. Yet in any projective line arrangement that has $G_0$ as dual graph, it is easy to see that either $r_1, r_3, r_4$ meet in a single point, or $r_2, r_3, r_4$ do. So a point of the line arrangement has multiplicity~$3$.
In contrast, Theorem \ref{thm:CMfromGraph} constructs always curve arrangements in which every singular point has multiplicity $2$.  
\end{remark}

We remind the reader that many graphs are neither line graphs, nor  dual to line arrangements. We have in fact the following hierarchy:

\begin{proposition}\label{prp:hierarchy} 
In any fixed dimension $d \ge 1$, we have \rm \small
\[
\left\{ \!\!
	\begin{array}{c}
	\textrm{line }\\
	\textrm{graphs} \\
	\end{array} \!\!
\right\}
\subset
\left\{ \!\!
	\begin{array}{c}
	\textrm{dual graphs }\\
	\textrm{of simplicial} \\
	d-\textrm{complexes}
	\end{array} \!\!
\right\}
\subsetneq
\left\{ \!\!
	\begin{array}{c}
	\textrm{dual graphs}\\
	\textrm{of projective} \\
	\textrm{line arr'ts}
	\end{array}\!\!
\right\}
\subsetneq
\left\{ \!\!
	\begin{array}{c}
	\textrm{dual graphs}\\
	\textrm{of affine} \\
	\textrm{line arr'ts}
	\end{array}\!\!
\right\}
\subsetneq
\left\{ \!\!
	\begin{array}{c}
	\textrm{dual graphs}\\
	\textrm{of projective}\\
		\textrm{ curves}
	\end{array}\!\!
\right\}
=
\left\{ \!\!
	\begin{array}{c}
	\textrm{all}\\
	\textrm{graphs}
	\end{array}\!\!
\right\}.
\]
\end{proposition}

\begin{proof} The 
\textsc{first inclusion} is obvious, and well-known to be 
strict if $ d \ge 2$, as shown for example by the complete bipartite graph $K_{1,3}$. (In fact, $K_{1,d+1}$ is the dual graph of some complex $C$ if and only if $\dim C \ge d$.)

\textsc{Second inclusion}: As explained in \cite{BV}, given any simplicial complex, its Stanley--Reisner variety is a subspace arrangement with same dual graph. Via generic hyperplane sections, we can reduce ourselves in turn from the Stanley--Reisner variety to a line arrangement with same dual graph. This proves the inclusion. 
The second inclusion is not an equality: the graphs $G_1$ and $G_2$ of Figure 1 are easy counterexamples. 
(For reasons of clarity, we postpone the proof of this fact, which requires combinatorial topology but is otherwise elementary, to the Appendix.)


\begin{figure}[htb]
\centering
\includegraphics[scale=.75]{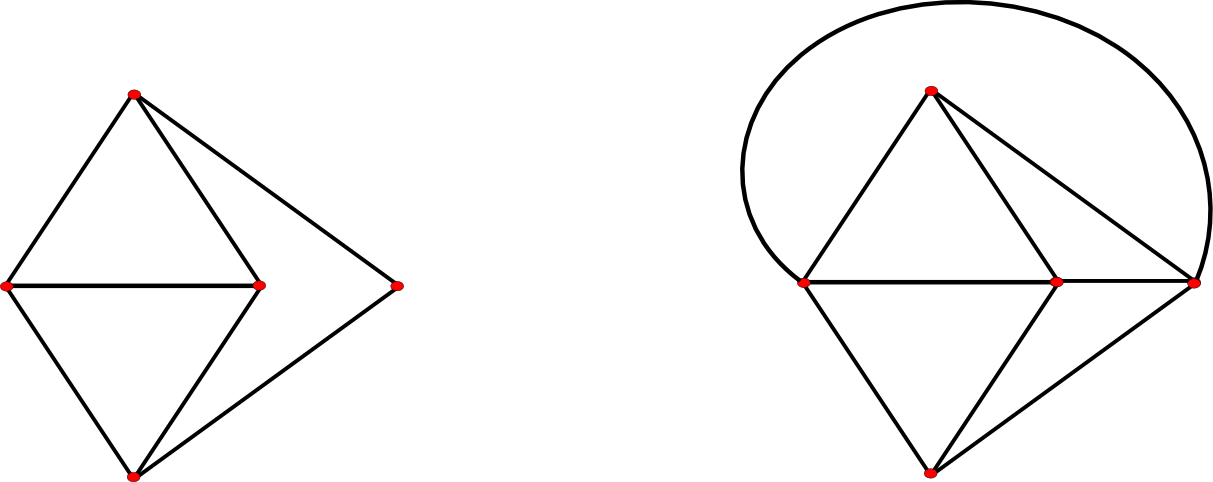} \small
\caption{Two graphs $G_1$, $G_2$ with $5$ vertices that are not dual to any simplicial complex of any dimension. 
}
\end{figure}

\textsc{Third inclusion}: Given a projective line arrangement $C\subset \PP^n$ we can always choose a hyperplane $H\subset \PP^n$ that avoids all the intersection points of $C$; if we set $U=\PP^n\setminus H$, then $C'=C\cap U  \subset U \cong \mathbb{A}^n$ is the desired affine line arrangement. As for the strictness: If $G_3$ is $K_6$ minus two non-adjacent edges, it is easy to produce a set of six affine lines with dual graph $G_3$; but in \cite[Rem.\ 4.1]{BV} we showed that $G_3$ cannot be dual to any projective line arrangement.

The \textsc{fifth inclusion} is well known to be an equality, see e.g.~Kollar \cite{Kollar}.

The \textsc{fourth inclusion} is therefore trivial; to prove its strictness, since the fifth inclusion is an equality, it suffices to find a graph that is not dual to any affine line arrangement. Consider the graph with the following edges
\[G_4 = \{12, 34\} \cup \{15, \, 25, \,35, \, 45 \}  \cup \{ 16, \, 26, \, 36, \, 46 \}  \cup \{ 17, \, 27, \, 37, \, 47 \}.\]  
By contradiction, let $\{r_1, \ldots, r_7 \}$ be an affine arrangement of lines with dual graph $G_4$. Set $P=r_1 \cap r_2$ and $Q=r_3 \cap r_4$. Let $\pi$ be the plane spanned by $r_1$ and $r_2$. Since $r_3$ intersects neither $r_1$ nor $r_2$, it cannot belong to the plane $\pi$. The same is true for $r_4$. Symmetrically, if $\mathfrak{q}$ is the plane spanned by $r_3$ and $r_4$, neither $r_1$ nor $r_2$ belongs to $\mathfrak{q}$. 
Now, how can a new line $r$ meet all four lines $r_1, r_2, r_3, r_4$? There are two options: 
\begin{compactitem}[ --]
\item either $r$ is the unique line passing through $P$ and $Q$, or 
\item $r$ is the (unique) line of intersection of the two planes $\pi$ and $\mathfrak{q}$. 
\end{compactitem}
But from the definition of $G_4$, we are supposed to find \emph{three} new lines ($r_5$, $r_6$ and $r_7$) each of which meets all four of $r_1, r_2, r_3, r_4$.  By the pidgeonhole principle, two of these three lines $r_5, r_6, r_7$ must coincide; a contradiction.
\end{proof}

\begin{remark}
Any graph containing $G_4$ as induced subgraph cannot be dual to any affine arrangement.
\end{remark}

\vskip3mm

\subsection*{Hochster--Huneke graphs and Lyubeznik complexes.} 
Let $A$ be a $d$-dimensional standard graded $\K$-algebra. Let $\{\pp_1,\ldots ,\pp_s\}$ be its minimal primes. We define the graph $G(A)$ as the graph whose vertices are $\{1,\ldots ,s\}$ and such that $\{i,j\}$ is an edge if and only if $A/(\pp_i+\pp_j)$ has dimension $d-1$. Obviously, $G(A)=G(\Proj(A))$. The graph $G(A) $ is sometimes called the {\it Hochster-Huneke graph} of $A$, after the work \cite{HH}. In this language, Theorem \ref{thm:CMfromGraph} implies that: 
\begin{corollary} Any 
connected graph is the Hochster--Huneke graph of a reduced 2-dimensional Cohen-Macaulay standard graded $\K$-algebra. 
\end{corollary}
\noindent After completing at the irrelevant ideal, this yields an affirmative answer to the question raised in \cite{SS} of whether any graph is the Hochster-Huneke graph of a complete equidimensional local ring (indeed one can show that for a positively graded $\K$-algebra the dual graph does not change after completing at the irrelevant ideal, as it follows by the Equation (4) of \cite[Theorem 1.15]{Va}).

Theorem \ref{thm:CMfromGraph} is of interest also from the point of view of a recent result obtained in \cite{KLZ}: The {\it Lyubeznik complex} of $A$ is the simplicial complex $\Delta(A)$ on vertices $\{1,\ldots ,s\}$, where $\{i_0,\ldots ,i_k\}$ is a face if and only if $\dim(A/(\pp_{i_0}+\ldots +\pp_{i_k}))>0$. (The terminology is due to the paper \cite{Ly}, where the complex $\Delta(A)$ was introduced.) If
$\dim(A/(\pp_i+\pp_j+\pp_k))=0$ for all $i<j<k$ and $\dim(A/(\pp_i+\pp_j))>0 \iff \dim(A/(\pp_i+\pp_j))=d-1$ (e.g. when $\Proj(A)$ is a curve such that no three irreducible components meet at the same point), then $\Delta(A)=G(A)$. So, Theorem \ref{thm:CMfromGraph} implies 
that 
\begin{corollary}
Any 
connected graph is the Lyubeznik complex of a reduced 2-dimensional Cohen-Macaulay standard graded $\K$-algebra. 
\end{corollary}
\noindent On the other hand, the results of \cite{KLZ} imply that, among graphs, only trees can be the Lyubeznik complex of a $d$-dimensional standard graded Cohen-Macaulay $\K$-algebra with $d\geq 3$ (if $\K$ is separably closed).

\section{From curves to graphs} 
In this section we prove the main result of the paper. Our first goal is to establish a bound on the regularity of a projective scheme as the sum of the regularities of its primary components. Unfortunately, this goal is hopeless in general, as there are counterexamples already among surfaces (see Example \ref{ex:aldo}). However, here we prove it for curves (cf. Lemma \ref{lem:EGrevisited} below), and later we will show that this suffices.

\begin{lemma}[{essentially Caviglia \cite{Cav}}]\label{lem:subadditivity}
Let $I, J$ be graded ideals of $S$.
If the Krull dimension of $\mathrm{Tor}^S_1(S/I,S/J)$ is at most 1, then $\reg(I\cap J)\leq \reg I + \reg J$.
\end{lemma}

\begin{proof}
By a result of Caviglia \cite[Corollary 3.4]{Cav} we have \[\reg S/(I+J) \le \reg S/I \, + \, \reg S/J .\] 
From the short exact sequence
$0\rightarrow S/(I\cap J)\rightarrow S/I\oplus S/J \rightarrow S/(I+J)\rightarrow 0,$
we immediately obtain
\[\reg \, S/(I\cap J) \, \le \max\{ \reg S/I, \reg S/J, 1 + \reg S/(I+J) \}\le \reg S/I +\reg S/J + 1.\] 
This is equivalent to the claim, because $\reg S/H =\reg H -1$ for any graded ideal $H$.
\end{proof}

\begin{lemma} \label{lem:EGrevisited}
Let $C\subset \PP^n$ be a curve and $C=\cup_{i=1}^sC_i$ a primary decomposition of $C$, then:
\begin{compactenum}[\rm (1)]
\item $\reg C\leq \reg C_1 + \ldots + \reg C_s$.
\item If in addition $C$ is reduced and $\dim C_i=\dim C_j$ then
$\reg C\leq \deg C$.
\end{compactenum}
\end{lemma}

\begin{proof}
For any $t \in \{2,\ldots ,s\}$, set
$J_t \defeq I_{C_t} + \bigcap_{i=1}^{t-1}I_{C_i}$. 
The ideal $J_t$ defines a $0$-dimensional projective scheme,
so $\dim S/J_t \le 1$. Moreover, $J_t$ annihilates $\mathrm{Tor}^S_1(S/\cap_{i=1}^{t-1}I_{C_i},\,S/I_{C_t})$. 
So the Krull dimension of
$\mathrm{Tor}^S_1(S/\cap_{i=1}^{t-1}I_{C_i}, \, S/I_{C_t})$ is at most 1. By Lemma \ref{lem:subadditivity} and by induction on $t$, we obtain 
$\reg I_C \leq \sum_{i=1}^s \reg I_{C_i}$,
which proves part (1) of the claim.

As for part (2), note that $\deg C=\sum_{i=1}^s \deg C_i$. Since the Eisenbud-Goto conjecture holds for integral curves \cite{GLP}, we have
\[\reg C_i \le  \deg C_i -\dim(\mathrm{Span} \ C_i) \, +2.\]
Therefore $\reg C_i \leq \deg C_i$ if $\dim (\mathrm{Span} \ C_i) \ge 2$. On the other hand, if the dimension of $\mathrm{Span} \ C_i$ is 1, then $C_i$ is a line: so $\reg C_i=1=\deg C_i$. By part (1) of the claim, which we have already proved, we conclude that
\[\reg C\leq \sum_{i=1}^s \reg C_i \leq \sum_{i=1}^s \deg C_i=\deg C. \qedhere\]
\end{proof}

\begin{example}\label{ex:aldo}
The above lemma cannot be extended to dimension $>1$. The following example is due to Aldo Conca: Let $H\subseteq \PP^4$ be the plane $x_1=x_2=0$. Let $\pp$ be the kernel of the map from $\K[x_0,\ldots ,x_4]$ to $\K[a,b,c]$ given by:
\[
x_0 \mapsto a^3b,  \qquad x_1 \mapsto b^4, \qquad  x_2 \mapsto a^3c, \qquad
x_3  \mapsto abc^2,  \qquad x_4 \mapsto b^2c^2.
\] 
This $\pp$ is a prime ideal of height 2 and regularity 5. Let  $X\subseteq \PP^4$ be the surface it defines. One has $\reg X=5$, $\reg H=1$, but $\reg (X\cup H)=7$. One can see (via Lemma \ref{lem:EGrevisited}) that any general hyperplane section of $X\cup H$ has regularity smaller than $7$. This is because $X\cup H$ is not arithmetically Cohen-Macaulay. 
\end{example}

\begin{theorem} \label{thm:1} Let $r, r'$ be positive integers. Let $X\subset \PP ^n$ be an arithmetically Gorenstein scheme of regularity $r+1$ and let $X=X_1\cup X_2 \cup \ldots \cup X_s$ be the primary decomposition of $X$. If $\reg X_i\le r'$ for all $i=1,\ldots ,s$, then the dual graph $G(X)$ is  $\lfloor{\frac{r+r'-1}{r'}}\rfloor$-connected.
\end{theorem}

\begin{proof}
First we show that there is no loss in assuming  that $X$ is a curve. In fact, if $\dim X \ge 2$, as explained in \cite[Lemma 2.12]{BV} we can always take a general hyperplane section of $X$, thereby obtaining a scheme $X'\subset \PP^{n-1}$ of dimension one less, such that $G(X)=G(X')$. Since $X$ is arithmetically-Gorenstein and the hyperplane section is general, in passing from $X$ to $X'$ both the arithmetically-Gorenstein property and the (global) Castelnuovo\---Mumford regularity are maintained. Caveat: the regularities of the components need not be maintained; but since the regularity of a general hyperplane section of any projective scheme cannot be larger than the regularity of the original scheme \cite[Lemma 4.8 + Corollary 4.10]{Eisenbud}, the regularities of the components of $X'$ can only be smaller or equal than those of $X$, so they will still be bounded above by $r'$. Iterating this process $\dim X -1$ times, we can reduce ourselves to the $1$-dimensional case.

By the assumption, $S/I_X$ is a Gorenstein ring of regularity $r$ and $\reg I_{X_i} \leq r'$ for all $i=1,\ldots ,s$. 
Let $B$ be a subset of $\{1,\ldots ,s\}$ of cardinality  $|B| < \lfloor{\frac{r+r'-1}{r'}}\rfloor$. Let $A=\{1,\ldots ,s\}\setminus B$. Define 
\[\begin{array}{lll}
X_A \defeq \bigcup_{i\in A}X_i , \\
X_B \defeq \bigcup_{i\in B}X_i. 
\end{array}\]

Our goal is to show that the dual graph of $X_A$ is connected, or in other words, that $X_A$ is a connected curve. 
(This would imply the claim, because the dual graph of $X_A$ is exactly the dual graph of $X$ with the vertices in $B$ removed, and $B$ was an arbitrary  subset of $\{1,\ldots ,s\}$ of cardinality less than $\lfloor{\frac{r+r'-1}{r'}}\rfloor$.)

The curves $X_A$ and $X_B$ are geometrically linked by $X$, which is arithmetically Gorenstein. Exploiting that $\dim X_i=1$ for all $i=1,\ldots ,s$, both $X_A$ and $X_B$ are locally Cohen-Macaulay curves. By Schenzel's work~\cite{Sc} (see also Migliore \cite[Theorem 5.3.1]{Migliore}) this implies the existence of a graded isomorphism
\[ H^1_{\mathfrak{m}}(S/ I_{X_A}) \cong H^1_{\mathfrak{m}}(S/ I_{X_B})^{\vee }(2-r) . \]
In particular the two finite $\K$-vector spaces $H^1_{\mathfrak{m}}(S/ I_{X_A})_0$ and $H^1_{\mathfrak{m}}(S/ I_{X_B})_{r-2}$ are dual to one another.
Notice that the connectedness of $C_A$ (which is what we want to show) follows by the vanishing of $H^1_{\mathfrak{m}}(S/ I_A)_0$, and thus of $H^1_{\mathfrak{m}}(S/ I_B)_{r-2}$. 
So it is enough to show that \[\reg S/I_{X_B} \le r-2.\]  
But $\reg S/I_{X_B} = \reg X_B-1$, and by our Lemma \ref{lem:EGrevisited} we have precisely
\[\reg X_B \le |B| \cdot r' \leq \left(\left\lfloor \frac{r+r'-1}{r'}\right\rfloor  - 1\right) r'   \le  \left(\frac{r+r'-1}{r'} - 1\right) r' = r-1. \qedhere\]
\end{proof}

One might wonder if the previous statement still holds by replacing the ``bounded regularity'' assumption for the primary components, with a ``bounded degree'' assumption. The answer is positive, but an additional assumption is needed, namely, $X$ should be reduced.

\begin{corollary} \label{cor:1degreeversion} Let $D$ and $r$ be positive integers. Let $X\subset \PP ^n$ be a reduced arithmetically Gorenstein scheme of regularity $r+1$. If every irreducible component of $X$ has degree $\le D$, then the dual graph $G(X)$ is  $\lfloor{\frac{r+D-1}{D}}\rfloor$-connected.
\end{corollary}

\begin{proof}
If $X$ has dimension $N$, let us take $N-1$ general hyperplane sections and call $X'$ the resulting curve. As in the proof of Theorem \ref{thm:1}, $X'$ is arithmetically Gorenstein and $G(X)=G(X')$. Since general hyperplane sections maintain the degree, each irreducible component of $X'$ has degree $\le D$. Moreover, $X'$ is reduced, since $X$ is. By \cite{GLP}, we infer that each irreducible component of $X'$ has also regularity $\le D$. Applying Theorem \ref{thm:1}, we conclude.
\end{proof}

\begin{corollary}\label{cor:ci}
Let $X_1,\ldots ,X_s\subseteq \PP^n$ be integral projective curves of degrees $d_1\leq d_2\leq \ldots \leq d_s$. If $X=\bigcup_{i=1}^sX_i\subseteq \PP^n$ is a complete intersection, then each $X_i$ must meet at least $\lfloor \frac{N+d_s-1}{d_s}\rfloor$ of the other $X_j$'s, where
\[N=\min\left\{\sum_{k=1}^{n-1}\delta_k-n+1 \ : \ \delta_i\in\mathbb{N} \mbox{ \ and \ }\prod_{k=1}^{n-1}\delta_k=\sum_{i=1}^sd_i\right\}\geq (n-1)\cdot \left( \left(\sum_{i=1}^sd_i \right)^{1/(n-1)}-1\right).\]
\end{corollary}
\begin{proof}
If $X\subseteq \PP^n$ is defined by $n-1$ equations of degrees $\delta_1,\ldots ,\delta_{n-1}$, then it is arithmetically Gorenstein of regularity $\sum_{k=1}^{n-1}\delta_k-n+2$, and $\prod_{k=1}^{n-1}\delta_k=\sum_{i=1}^sd_i$. So $G(X)$ is $\lfloor \frac{N+\max_i\{d_i\}-1}{\max_i\{d_i\}}\rfloor$-connected by Corollary \ref{cor:1degreeversion}. In particular, each vertex of $G(X)$ has valency at least $\lfloor \frac{N+\max_i\{d_i\}-1}{\max_i\{d_i\}}\rfloor$.
\end{proof}

\begin{example}[27-lines]
With the notation of Corollary \ref{cor:ci}, if $s=27$, $d_i=1$ and $n=3$, then $N=10$. So if a union of 27 lines in $\PP^3$ happens to be a complete intersection, then each line must meet at least 10 of the others by Corollary \ref{cor:ci}. 

It is easy to see that the union of the 27 lines in a smooth cubic surface in $\PP^3$ is a complete intersection (the cubic cut out by a union of 9 planes); one can see that any of the 27 lines meet {\it exactly} with 10 of the others. In this sense Corollary \ref{cor:ci} is sharp.
\end{example}

In case the regularities of the irreducible components are quite diverse (for example, if one component has much larger regularity than all the other ones) the following sharpening of Theorem \ref{thm:1} could be convenient:

\begin{theorem} \label{thm:1general} Let $r$ be a positive integer. Let $X\subset \PP ^n$ be an arithmetically Gorenstein scheme of regularity $r+1$ and $X=X_1\cup X_2 \cup \ldots \cup X_s$ be the primary decomposition of $X$. Then the dual graph $G(X)$ is $f(r)$-connected, where 
\[f(r) \defeq \max\; \{i \in \mathbb{N}  \textrm{ s.t. for all } B \subset [s] \textrm{ of cardinality } i-1, \textrm{ one has } \sum_{j\in B} \reg X_j \le r -1 \} . \]
\end{theorem}

The proof is the same of Theorem \ref{thm:1}.  Of course, the analogous sharpening of Corollary \ref{cor:1degreeversion} holds.

\begin{remark}
Suppose an arithmetically Gorenstein scheme $X$ has at least two primary components. It is natural to ask whether the regularity of the components is bounded above by $\reg X$. The answer is:
\begin{compactitem}
\item positive for the components $Q$ that are arithmetically Cohen--Macaulay (one has $\reg Q \le \reg X -1$);
\item ``very'' negative in general, even if $X$ is a complete intersection, as the example below shows.
\end{compactitem}
Let us consider the complete intersection 
\[I = (x_0 x_1^{2000}-x_4 x_2^{2000}, \; \: x_0 x_1 x_3+x_4^3, \: \; x_0^3+2x_1 x_3^2)   
\; \subset \; S \, = \, \mathbb{Q}[x_0, \ldots, x_4].
\]
With Macaulay2 \cite{macaulay2} we can see that the primary decomposition of $I$ consists of three ideals:
\[Q_1=\left(
\begin{array}{l}
{x}_{0} {x}_{1} {x}_{3}+{x}_{4}^{3}, \:\;
{x}_{0}^{3}+2 {x}_{1} {x}_{3}^{2}, \:\;
{x}_{2}^{2000} {x}_{3}+{x}_{1}^{1999} {x}_{4}^{2}, \:\;
{x}_{0} {x}_{1}^{2000}-{x}_{2}^{2000} {x}_{4}, \\
{x}_{0}^{2} {x}_{2}^{4000}-2 {x}_{1}^{4000} {x}_{3} {x}_{4}, \:\;
{x}_{0} {x}_{2}^{6000}-2 {x}_{1}^{6000} {x}_{3}, \:\;
{x}_{2}^{10000}+2 {x}_{1}^{9999} {x}_{4} 
\end{array}\right),
\]
\[ Q_2=\left({x}_{4},{x}_{1}^{2},{x}_{0} {x}_{1},{x}_{0}^{3}+2 {x}_{1} {x}_{3}^{2}\right), 
\phantom{lalalalalalalalalalalalalalalalalalalalalalalaa}
\]
\[ Q_3= \left(
\begin{array}{l}
{x}_{3} {x}_{4}^{2}, \:\:
{x}_{3}^{2} {x}_{4}, \:\: 
{x}_{0} {x}_{3} {x}_{4}, \:\:
{x}_{3}^{3}, \:\:
{x}_{0} {x}_{3}^{2}, \:\:
{x}_{0} {x}_{1} {x}_{3}+{x}_{4}^{3},\:\:
{x}_{0}^{2} {x}_{3},\:\:
{x}_{0}^{3}+2 {x}_{1} {x}_{3}^{2},\\
{x}_{0}^{2} {x}_{4}^{2}, \:\:
{x}_{0} {x}_{1}^{2000}-{x}_{2}^{2000} {x}_{4},\:\:
{x}_{0} {x}_{2}^{2000}
      {x}_{3}+{x}_{0} {x}_{1}^{1999} {x}_{4}^{2}
      \end{array}\right).
      \]
The dual graph is $K_3$. While the regularity of $I$ is $2005$, the regularities of the $Q_i$'s are $10\, 000$, $3$ and $2003$, respectively. In fact, $S/Q_1$ is not Cohen--Macaulay.
\end{remark}

\subsection{Sharpness of the bound} \label{sec:sharpness}
The $\lfloor{\frac{r+r'-1}{r'}}\rfloor$ connectivity bound given by Theorem \ref{thm:1} is sharp for $r'=1$, as already noticed in \cite{BV}. In this section we prove that such bound is sharp also for some $r' > 1$. This answers a question by Michael Joswig (personal communication). 
To this end, we focus on arithmetically Gorenstein curves in $\PP^4$, for which a routine to generate examples is available. 

Let $k \ge 2$ be an integer and let $n=2k+1$. Let $M$ be an $n \times 
n$ upper triangular matrix of homogenous polynomials of degree $d$ in $\mathbb{Q}[x_0, \ldots, x_4]$. 
Let $A=M - M^t$. By definition, $A$ is skew-symmetric. Let $I$ be the ideal generated by the pfaffians of size $n-1$ of $A$. By a result of Eisenbud--Buchsbaum, if $\operatorname{height} I = 3$ then $S/I$ is Gorenstein of regularity $dn-3$. Of course, the ideal $I$ is completely determined by the matrix $M$, and for this reason we will denote it by $I_M$.

\begin{example} \label{ex:1}
Consider the upper triangular matrix
\[M = \left(
\begin{array}{ccccc}
	0 & x_1^2+x_4 x_3 & 0 & 0 &x_3 x_5+x_1^2+x_3 x_4 \\
	0 & 0 & x_5 x_4-x_2^2 & 0 & x_1 x_3-x_4^2 \\
	0 & 0 & 0 & x_3^2+x_5 x_1 & x_5^2+x_3 x_1-x_2 x_4\\
	0 & 0 & 0 & 0 & x_5 x_2-x_3^2 \\
	0 & 0 & 0 & 0 & 0
\end{array} \right).
\]
Using Macaulay2 \cite{macaulay2}, we computed the ideal $I_M$. The regularity of $S/I_M$  is $2 \cdot 5 - 3= 7$. There are eight primary components of $I$, of regularities 
\[3, \; 2, \; 2, \; \mathbf{6}, \; 3, \; 4, \; 4, \textrm{ and } 3,\] with respect to the default ordering used by the software, namely, the graded reverse lexicographic order.
So the maximum regularity is $6$. According to Theorem \ref{thm:1}, the dual graph is  $\lfloor{\frac{7+6-1}{6}}\rfloor$-connected, that is, $2$-connected. In fact, the dual graph has $8$ vertices and edge list 
\[G_5 =\{ 12, 14, 23, 24, 27, 34, 35, 36, 37, 38, 45, 46, 47, 48, 56, 57, 58, 67, 68, 78\}.\]
Indeed this $G_5$ is not $3$-connected, as the vertex labeled by $1$ has degree two. So the bound given by Theorem \ref{thm:1} is best possible on this example.
\end{example}

\begin{example}\label{ex:2}
Consider the upper triangular matrix
\[M' = \left(
\begin{array}{ccccc}
	0 & x_1^2+x_4 x_3 & 0 & 0 &x_3 x_5+x_1^2+x_3 x_4 \\
	0 & 0 & x_5 x_4-x_2^2 & 0 & x_1 x_3+x_4^2 \\
	0 & 0 & 0 & x_3^2+x_5 x_1 & x_5^2+x_3 x_1-x_2 x_4\\
	0 & 0 & 0 & 0 & x_5 x_2-x_3^2 \\
	0 & 0 & 0 & 0 & 0
\end{array} \right).
\]
This is almost identical to the matrix $M$ of Example~\ref{ex:1}! The only change is a minus sign that has become a plus, in the fifth binomial of the second row. Again using Macaulay2 \cite{macaulay2}, we computed the ideal $I_{M'}$ associated. The regularity of $S/I_{M'}$  is $7$; the eight primary components of $I_{M'}$ have regularities 
\[3, \; 2, \; 2, \; \mathbf{7}, \; 3, \; 4, \;4, \textrm{ and } 3.\] 
So the maximum regularity is now $7$. Note that with respect to Example \ref{ex:1} all the regularities are unchanged, except the maximal one, which has increased by one. This time, with Theorem \ref{thm:1} we can only say that the dual graph is  $\lfloor{\frac{7+7-1}{7}}\rfloor$-connected, that is, $1$-connected. And as a matter of fact, the dual graph is now 
\[G_6=\{14, 23, 24, 27, 34, 35, 36, 37, 38, 45, 46, 47, 48, 56, 57, 58, 67, 68, 78 \},\]
which is $1$-connected, but not $2$-connected, because the vertex labeled by $1$ has now degree one. (In fact, $G_6$ is the $G_5$ of  Example \ref{ex:1} with the edge $12$ deleted.) 
\end{example}

These examples show that the bound given by Theorem \ref{thm:1} is best possible, and quite sensitive to minimal variations in the regularity. Of course, the sharpness of a bound on some examples, does not imply that the bound is sharp on \emph{all} examples. There are plenty of arithmetically Gorenstein curves whose dual graph is much more connected than what Theorem \ref{thm:1} enables us to see.

\section*{Appendix}
Here we show that some dual graphs of (affine or projective) line arrangements are not dual graphs of any simplicial complex of any dimension. The proof of this fact is quite elementary, but we decided to include it for completeness.
Let us start with some easy notation. By convention, 
\begin{compactitem}[--]
\item the $(-1)$-simplex is just the empty set;  
\item the join of the empty set with a complex $C$, is $C$ itself.
\end{compactitem}
\begin{definition}[$d$-star, $d$-windmill]
Let $d \ge 1$. \\A \emph{$d$-star} is the complex obtained by joining the graph $K_3$ with a $(d-2)$-simplex. \\A \emph{$d$-windmill} is the complex obtaned by joining the graph $K_{1,3}$ with a $(d-2)$-simplex. (Equivalently, a $d$-windmill is the join of 3 disjoint points with a $(d-1)$-simplex). 
\end{definition}

\begin{figure}[hbtp] \label{fig:1}
\centering
\includegraphics[scale=.6]{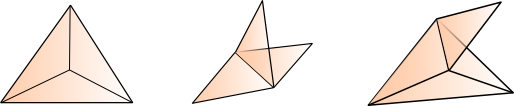}
\caption{The $2$-star, the $2$-windmill, and the unique pure $2$-complex with the diamond graph as dual.}
\end{figure}

\begin{lemma} \label{lem:StarOrWindmill}
For any fixed positive integer $d$, there are exactly \emph{two} pure $d$-complexes with dual graph $K_3$, namely, the $d$-star and the $d$-windmill.
\end{lemma}

\begin{proof} The dual graph does not change under taking cones, and in particular, it does not change under taking joins with simplices. Hence, the $d$-star and the $d$-windmill both have $K_3$ as dual graph (because both $K_3$ and $K_{1,3}$ have $K_3$ as dual graph).  Conversely, let $C$ be a pure $d$-complex with dual graph $K_3$; we claim that $C$ is either a $d$-star or a $d$-windmill. 
If $d=1$, the claim is clear. So assume $d \ge 2$. Let $\Sigma_1$, $\Sigma_2$ and $\Sigma_3$ be the $d$-simplices of $C$. Set 
\[I =  \Sigma_1 \cap \Sigma_2  \cap \Sigma_3.\] 
Let us denote by $c$ (respectively, by $i$) the total number of vertices of $C$ (respectively, of $I$).
Since two of the $\Sigma_i$ have exactly $d$ vertices in common, by the inclusion-exclusion formula we have  
\[ c = 3 (d+1) - 3d + i = 3 + (\dim I +1). \]
But $c \ge d+2$, otherwise $C$ would consist of a single simplex. So $\dim I \ge d-2$ and therefore $\dim \operatorname{link}(I,C) \le 1$.
So $\operatorname{link}(I,C)$ is a graph with dual graph $K_3$; hence, it must be either $K_3$ itself, or a disjoint union of three points. Since $C$ is the join of $\operatorname{link}(I,C)$ with a simplex, the claim follows.
\end{proof}

\begin{lemma} \label{lem:UniqueDual}
Let $G$ be the \emph{diamond graph}, that is, $K_4$ minus an edge. In any fixed dimension $d \ge 1$, there is exactly \emph{one} pure simplicial $d$-complex with dual graph $G$. 
This complex is obtained from a $d$-star by glueing in a further $d$-simplex to one of the internal $(d-1)$-faces of the $d$-star. In particular, the two non-adjacent $d$-simplices of the complex share exactly $d-1$ vertices.
\end{lemma}

\begin{proof}
Let $E$ be a complex with dual graph $\{12, 13, 23, 24, 34\}$.  Let $\Sigma_i$ be the facet of $C$ corresponding to $i$ ($i=1, \ldots, 4$). Let $C$ be the subcomplex of $E$ induced by the facets $\Sigma_1, \Sigma_2$ and $\Sigma_3$. By Lemma \ref{lem:StarOrWindmill}, $C$ is either a $d$-star or a $d$-windmill. 
If $C$ is a $d$-star, then $E$ must be obtained by attaching $\Sigma_4$ to the $(d-1)$-face $\Sigma_2 \cap \Sigma_3$ of $C$. In particular $\Sigma_1 \cap \Sigma_4 = \Sigma_1 \cap \Sigma_2 \cap \Sigma_3$ consists of $d-1$ vertices, and the claim is proven.
If instead $C$ is a $d$-windmill, by definition there are three vertices $v_1, v_2, v_3$ and a $(d-1)$-simplex $\sigma$ such that $\Sigma_i  = v_i * \sigma$ for each $i$. Consider the 
subcomplex $D$ of $E$ formed by $\Sigma_2, \Sigma_3$ and $\Sigma_4$. The dual graph of $D$ is $K_3$, so by Lemma \ref{lem:StarOrWindmill}, $D$ is either a $d$-star or a $d$-windmill; but it cannot be a $d$-windmill, otherwise the common intersection $\Sigma_2 \cap \Sigma_3 \cap \Sigma_4$ would have to be $\sigma$, which is contained in $\Sigma_1$: a contradiction, in $G$ there is no edge $14$. So $E$ is obtained from $D$ by glueing the simplex $\Sigma_1$ onto the internal face $\Sigma_2 \cap \Sigma_3$ of the $d$-star $D$.
\end{proof}

\begin{lemma} \label{lem:GraphExample}
No simplicial complex of any dimension has the graph $G_1$ of Figure 1 as dual graph.
\end{lemma}

\begin{proof}
By contradiction, let $F$ be a $d$-dimensional complex with dual graph 
\[ G_1 = \{12, 13, 15, 23, 24, 34, 45\}.\] 
Let $\Sigma_i$ be the facet of $F$ corresponding to $i$ ($i=1, \ldots, 5$), and let $E$ be the subcomplex induced by the first four facets. The dual graph of $E$ coincides with the diamond graph $G$ of Lemma \ref{lem:UniqueDual}, so we know how $E$ looks like. In particular $\Sigma_1$ and $\Sigma_4$ have $d-2$ vertices in common.  In the rest of the proof we give a formal reason why the way the pattern of adjacencies of the fifth simplex yields a contradiction. (We invite the reader to check this in Figure \ref{fig:1}, by trying to place a fifth triangle $\Sigma_5$ in the rightmost complex so that it has edges in common only with the two triangles that are not adjacent to one another.) 

Let us agree on some notation first. 
 Without loss of generality, we can assume the subcomplex induced by the first three facets $\Sigma_1, \Sigma_2, \Sigma_3$ is a $d$-star, and the complex induced by $\Sigma_2, \Sigma_3, \Sigma_4$ is a $d$-windmill. (If not, we switch the labels of $\Sigma_1$ and $\Sigma_4$.) By definition of windmill, if 
$\sigma = \Sigma_2 \cap \Sigma_3$, there are three vertices $v_2, v_3, v_4$ such that
$\Sigma_i$ ($i=2,3,4$) is of the form $v_i \ast \sigma$. Also, there is exactly one vertex $v$ of $\sigma$ that does \emph{not} belong to $\Sigma_1$. With this notation, the list of vertices of 
$\Sigma_1$ and that of $\Sigma_4$ have an overlap of exactly $d-1$ vertices. In fact, the lists differ only in the following:
\begin{compactitem}
\item $v_2$ and $v_3$ only belong to $\Sigma_1$;
\item $v$ and $v_4$ only belong to $\Sigma_4$. 
\end{compactitem}
But the new simplex $\Sigma_5$ is adjacent to \emph{both} $\Sigma_1$ and $\Sigma_4$.  This means that when we compare the list of vertices of $\Sigma_1$ with the list of vertices of $\Sigma_5$, we see only one change, and this change is one of the following four: 
\begin{compactenum}[(a)]
\item $v_2$ is replaced by $v$; 
\item $v_2$ is replaced by $v_4$;
\item $v_3$ is replaced by $v$;
\item $v_3$ is replaced by $v_4$.
\end{compactenum}
In cases (a) (resp. (c) ), $\Sigma_5$ has then the same set of vertices of $\Sigma_3$ (resp. of $\Sigma_2$), a contradiction. In cases (b) (resp. (d)), $\Sigma_5$ would be adjacent to $\Sigma_3$ (resp. to $\Sigma_2$), also a contradiction.
\end{proof}

Similarly one can show that the graph $G_2$ of Figure 1 cannot be the dual graph of any simplicial complex, either. We leave it to the reader to construct two projective line arrangements with the graphs $G_1$ and $G_2$ as dual.

\subsection*{Acknowledgments}
Many thanks to Eric Katz for very valuable suggestions; thanks also to Tony Iarrobino for discussions. Part of the research was conducted during a visit of the first author at Northeastern University, for which the author wishes to thank Alex Suciu.

\begin{small}

\end{small}
\end{document}